\documentclass[english]{article}
\usepackage{graphicx} 
\usepackage{tikz}
\usepackage{tikz-3dplot}
\usepackage{pgfplots}
\usepackage{geometry}
\usepackage{multirow}
\usepackage{diagbox}
\usepackage{amsmath, bm}
\usepackage{amsthm}
\usepackage{amssymb}
\usepackage{booktabs}
\usepackage{varwidth}
\usepackage{array}
\usepackage{tabularray}
\PassOptionsToPackage{normalem}{ulem}
\usepackage{ulem}
\usepackage{xurl}
\makeatletter
\usepackage{xcolor}
\usepackage{color}
\usepackage{algorithm}
\usepackage[noend]{algpseudocode}
\usepackage{graphicx}
\usepackage[utf8]{inputenc}
\usepackage[font=small, labelfont=bf]{caption}
\usepackage{float}
\usepackage{cancel}

\usepackage[hidelinks]{hyperref}
\usepackage[style=ieee,backend=biber,doi=true,url=false,giveninits=true]{biblatex}

\DeclareFieldFormat{doi}{\href{https://doi.org/#1}{doi:\nolinkurl{#1}}}

\DeclareNameAlias{default}{given-family}
\DeclareNameAlias{sortname}{given-family}

\makeatother

\usetikzlibrary{calc,arrows.meta,patterns}

\definecolor{myblue}{RGB}{45,130,220}
\definecolor{myorange}{RGB}{240,140,40}

\newtheorem{lem}{\protect\lemmaname}
\newtheorem{thm}{\protect\theoremname}

\newtheorem{rem}{\protect\remarkname}

\providecommand{\theoremname}{Theorem}
\providecommand{\remarkname}{Remark}

\makeatother

\providecommand{\lemmaname}{Lemma}

\usepackage{algpseudocode}
\newlength\fwidth
\makeatletter
\renewcommand{\ALG@beginalgorithmic}{\small}

\algrenewcommand\textproc{} 
\makeatother

\providecommand{\keywords}[1]
{
  \small	
  \textbf{\textit{Keywords---}} #1
}

\newcommand{\cp}{\mathrm{cp}}
\newcommand{\cpbar}{\bar{\mathrm{cp}}}
\newcommand{\cpx}{\cp(\bm{x})}
\newcommand{\cpbarx}{\bar{\cp}(\bm{x})}
\newcommand{\B}{\mathcal{B}}
\newcommand{\I}{\mathcal{I}}
\newcommand{\E}{\mathcal{E}}
\renewcommand{\S}{\mathcal{S}}

\DeclareMathOperator*{\argmin}{arg\,min}

\graphicspath{{figures/}}

\numberwithin{equation}{section}

\let\oldnl\nl
\newcommand{\nonl}{\renewcommand{\nl}{\let\nl\oldnl}}
\makeatother

\begin{document}
\title{The Closest Point Method for Surface PDEs with General Boundary Conditions}
\author{Tony Wong\thanks{Department of Mathematics, University of
California, Los Angeles, USA} \and
Colin B. Macdonald\thanks{Department of Mathematics, The University of British Columbia, Canada} \and
Byungjoon Lee\thanks{Department of Mathematics, The Catholic University of Korea, Republic of Korea}\textsuperscript{\ \ ,\P }}
\maketitle
\begingroup
\renewcommand\thefootnote{\fnsymbol{footnote}}
\footnotetext[5]{Corresponding author. Email: blee@catholic.ac.kr}
\endgroup
\maketitle

\begin{abstract}
We generalize the closest point method (CPM) to solve surface partial differential equations with general boundary conditions. The proposed extrapolation method provides a unified framework for treating a broad class of inhomogeneous Neumann and Robin boundary conditions within the framework of CPM. The accuracy and robustness of the method are demonstrated through numerical convergence studies of an elliptic problem, Steklov eigenvalue problems, and a nonlinear reaction–diffusion system.
\end{abstract}


\section{Introduction}

Many partial differential equations (PDEs) arising in scientific and engineering problems are defined on curved surfaces, which requires robust and accurate numerical methods to resolve the geometry. The closest point method (CPM) \cite{ruuth2008simple} is a versatile embedding method for numerically solving PDEs on curved geometry, with applications such as pattern formation \cite{macdonald2013simple, charette2020pattern, marz2016embedding,rozada2014stability}, geometry and image processing \cite{tian2009segmentation, auer2012real, biddle2013volume, king2024closest, arteaga2015laplace}, and first-passage time problems of diffusion \cite{iyaniwura2021simulation, iyaniwura2021optimization}. 

However, there is limited discussion on the uses of CPM for open surfaces.
It remains an open challenge to extend the CPM to surface PDEs with general boundary conditions beyond the homogeneous Neumann and Dirichlet types \cite{macdonald2011solving}.
In this paper, we provide a simple and accurate extrapolation technique for the CPM to tackle a broad class of boundary conditions.
Our work is particularly relevant for applications where boundary effects are critical, such as fluid dynamics on curved interfaces and biological membrane modeling.

\subsection{Terminology and notations}

Let $\S$ be an open surface --- a smooth surface embedded in $\mathbb{R}^3$, with a co-dimension one boundary $\partial \S$. For $\bm{x} \in \S$, let $\mathbf{N}(\bm{x})$ be a unit normal to $\S$.  For $\bm{y} \in \partial \S$, let $\mathbf{T}(\bm{y})$ a unit tangent to $\partial \S$, and $\mathbf{n}(\bm{y})$ the unit outward co-normal to $\partial \S$, as in Figure~\ref{fig:terminology}. Let $\nabla_S$ be the surface gradient on $\S$ and $\partial_n =\mathbf{n} \cdot \nabla_\S $ be the co-normal derivative defined on $\partial\S$. In this paper, we consider the following boundary condition for a function $u$ defined on $\S$:
\begin{equation}\label{eqn:gbc}
\partial_n u \big\vert_{\bm{y}} = j\left(\bm{y},u(\bm{y})\right), \qquad \bm{y} \in \partial\S.
\end{equation}
The form \eqref{eqn:gbc} includes many commonly used boundary conditions.
For example, the homogeneous Neumann condition corresponds to $j = 0$, while the Robin condition corresponds to $j = -\kappa u$.

The \emph{closest point function} of $\S$ is defined by  $\cpx := \argmin_{\bm{y} \in \S} \| \bm{x} - \bm{y} \|$, where $\|\cdot\|$ denotes the Euclidean norm. Provided that $\S$ is smooth, there exists a tubular neighbourhood $\B(\S)$ where the closest point function $\cpx$ is well-defined \cite{hirsch2012differential}. Within $\B(\S)$, we define the set of \emph{interior points} as $\I(\S) = \{ \bm{x} \in \B(\S) : \cpx \notin \partial \S \}$ and the set of \emph{exterior points} as $\E(\S) = \{ \bm{x} \in \B(\S) : \cpx \in \partial \S \}$ (see Figure~\ref{fig:terminology}(b)).


The \emph{closest point extension operator} $E$ extends a surface function $u$ to $\B(\S)$ by $E u(\bm{x}) := u(\cpx)$. The operator $E$ propagates the surface values along the normal direction in $\I(\S)$ and the boundary values to $\E(\S)$.

Following \cite{macdonald2011solving}, we define the \emph{modified closest point function} $\cpbarx := \cp(r(\bm{x}))$,
where $r(\bm{x})$ is the reflection of $\bm{x}$ across $S$,
i.e., $r(\bm{x}) := 2\,\cpx - \bm{x}$.
The \emph{modified closest point extension operator} $\bar{E}$ is defined by $\bar{E}u(\bm{x}) := u(\cpbar(\bm{x}))$.
If $\bm{x} \in \I(\S)$, we have $\cpbarx=\cpx$.
Therefore, the extensions $\bar{E}u$ and $Eu$ coincide in $\I(\E)$. For $\bm{x} \in \E(\S)$, the value of $\bar{E}u(\bm{x})$ is $\bar{E}u(r(\bm{x}))$.
In other words, $\bar{E}$ mirrors the interior values to the exterior points, and indeed this is sufficient to implement homogeneous Neumann conditions \cite{macdonald2011solving}, but is not enough for a smooth implementation of \eqref{eqn:gbc}.



\begin{figure}[t]
\centering
    \includegraphics[width=0.9\linewidth]{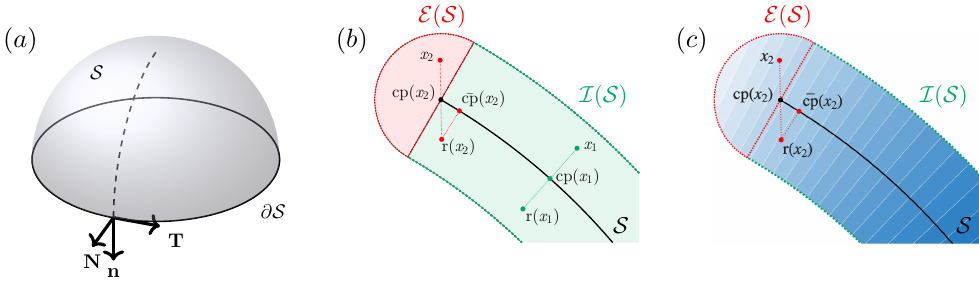}
    \caption{(a) Hemisphere $\mathcal{S}$ with boundary $\partial \mathcal{S}$, normal $\mathbf{N}$, tangent $\mathbf{T}$, and outward co-normal $\mathbf{n}$. (b) An interior point $\bm{x}_1 \in \I(\S)$, we have $\bar{\cp}\left(\bm{x}_1 \right) = {\cp}\left(\bm{x}_1 \right)$. An exterior point $\bm{x}_2 \in \E(\S)$, which satisfies $\cp\left(\bm{x}_j\right) \in \partial \S$ and $\bar{\cp}\left(\bm{x}_j \right) \neq \cp\left(\bm{x}_j\right)$. (c) The smooth extension $\bar{u}$ defined in \eqref{eqn:extension}.}
    \label{fig:terminology}
\end{figure}





\section{The closest point method}

We introduce two lemmas to support our main theorem characterizing a smooth extension which propagates surface values constantly along the normal direction in $\I(\S)$ and into $\E(\S)$ according to the boundary condition \eqref{eqn:gbc}.
We then describe the numerical discretization.

\subsection{Extrapolation for the general boundary condition}


\begin{lem}\label{lem:span}
  Suppose $\bm{x} \in \E(\S)$.
  Let $\mathbf{N} = \mathbf{N}(\cp(\bm{x}))$ and $\mathbf{n} = \mathbf{n}(\cp(\bm{x}))$.
  Then 
  $\bm{x} - \cp(\bm{x})$ lies in $\mathrm{span}\{\mathbf{N}, \mathbf{n}\}$.
\end{lem}

\begin{proof}
Near $\cp(\bm{x})$, there exists a local chart $(U, \phi)$ mapping a neighborhood $U \subset \S$ to the half-plane $\{(\xi,\eta)\in \mathbb{R}^2 : \eta \ge 0\}$, with $\partial \S \cap U$ corresponding to $\eta = 0$. The boundary curve admits the 1D parameterization $\gamma(t) = \phi^{-1}(t,0)$ for $t$ in some interval $I$, and $\gamma$ is smooth and regular. Let $\cp(\bm{x}) = \gamma(t_*)$ with
$t_* = \argmin_{t \in I} \| \gamma(t) - \bm{x} \|^2$.
The first-order condition gives $\gamma'(t_*)^T [\gamma(t_*) - \bm{x}] = 0$, so $\bm{x} - \cp(\bm{x})$ is orthogonal to $\gamma'(t_*)$, the tangent to $\partial \S$ at $\cp(\bm{x})$. Since $\{\mathbf{N}, \mathbf{T}, \mathbf{n}\}$ is an orthogonal basis, it follows that $\bm{x} - \cp(\bm{x}) \in \mathrm{span}\{\mathbf{N}, \mathbf{n}\}$.
\end{proof}
With Lemma~\ref{lem:span}, we characterize how a smooth extension $\bar{u}$ propagates the boundary values into $\E(\S)$, according to the \eqref{eqn:gbc}.
\begin{lem}\label{lem:ihn_BC}
  Let $u$ be a smooth function defined on $\S$ that satisfies \eqref{eqn:gbc}.
  Let $\bar{u}$ be a smooth extension of $u$ in $\B(\S)$ such that $\bar{u}(\bm{x}) = u(\cpx)$ for $\bm{x} \in \I(\S)$.
  Suppose $\bm{x} \in \E(\S)$. Let $\mathbf{w} = (\bm{x}-\cpx)/\|\bm{x}-\cpx\|$.
  Denote the unit outward co-normal $\mathbf{n}(\cpx)$ by $\mathbf{n}$. Then, we have
\begin{equation}\label{eqn:boundary_derivative}
\partial_w \bar{u} \,\vert_{\cpx} = \langle \mathbf{w}, \mathbf{n} \rangle \, j\big(\cpx, \bar{u}(\cpx)\big) \,, \qquad \mbox{where $\langle \cdot, \cdot \rangle$ denotes the Euclidean inner product.}
\end{equation}
\end{lem}

\begin{proof}
By Lemma~\ref{lem:span}, $\mathbf{w} \in \mathrm{span}\{\mathbf{N}, \mathbf{n}\}$, where $\mathbf{N}$ is a unit normal at $\cpx$. We apply the orthogonal decomposition $\mathbf{w} = \langle \mathbf{w}, \mathbf{n} \rangle \mathbf{n} + \langle \mathbf{w}, \mathbf{N} \rangle \mathbf{N}$ and the linearity of directional derivatives to obtain
\[
\partial_w \bar{u} \,\vert_{\cpx} = \langle \mathbf{w}, \mathbf{n} \rangle \, \partial_n \bar{u} \,\vert_{\cpx} + \langle \mathbf{w}, \mathbf{N} \rangle \, \partial_N \bar{u} \,\vert_{\cpx}  \,.
\]
Now $\bar{u}$ is constant along the normal $\mathbf{N}$
so $\partial_N \bar{u} \vert_{\cpx} = 0$.
Secondly, we have $\partial_n \bar{u} \vert_{\cpx} = \partial_n u \vert_{\cpx} = j\big(\cpx, u(\cpx)\big) = j\big(\cpx, \bar{u}(\cpx)\big)$.
Hence, we obtain \eqref{eqn:boundary_derivative}.
\end{proof}




\begin{thm}\label{thm:BCext}
  Suppose $\S$ is a smooth surface with a tubular neighbourhood $\B(\S)$.
  Suppose $\bar{u}: \B(\S) \to \mathbb{R}$ satisfies the extrapolation formula for boundary points 
\begin{equation}\label{eqn:extrap}
\begin{aligned}
\bar{u}(\bm{x}) = \bar{u}(\cpbarx) &+ 2 \big\langle \bm{x} - \cpx, \mathbf{n}(\cpx) \big\rangle \, j\big(\cp(\bm{x}), \bar{u}(\cp(\bm{x}))\big) \\
= \bar{E} \bar{u}(\bm{x}) &+ 2 \big\langle \bm{x} - \cpx, \mathbf{n}(\cpx) \big\rangle \, j\big(\cpx, E \bar{u}(\bm{x}) \big),
\qquad \bm{x} \in \E(\S),
\end{aligned}
\end{equation}
and $\bar{u}(\bm{x}) = \bar{u}(\cpbarx)$ for interior points $\bm{x} \in \I(\S)$.
If $\bar{u}$ is a sufficiently-smooth function when restricted to surfaces points $\bm{y} \in \S$,
then $\bar{u}\vert_{\partial S}$ satisfies the boundary condition \eqref{eqn:gbc} to within $O(\|\bm{x} - \cpx\|^3)$.
\end{thm}


\begin{proof}
  Let $\bm{x} \in \E(\S)$.
  Applying the central difference formula to $\partial_w \bar{u}\vert_{\cpx}$ with the stencil $\{\bm{x}, r(\bm{x})\}$, we have
\begin{equation}\label{eqn:central_diff}
    \partial_w \bar{u}\vert_{\cpx} = \frac{\bar{u}(\bm{x}) - \bar{u}(r(\bm{x}))}{\|\bm{x} - r(\bm{x})\|} + O(\|\bm{x} - \cpx\|^2)
    = \frac{\bar{u}(\bm{x}) - \bar{u}(\cpbarx)}{2\|\bm{x} - \cp(\bm{x})\|} + O(\|\bm{x} - \cpx\|^2),
\end{equation}
where the second equality follows by $\bar{u}(r(\bm{x})) = \bar{u}(\cpbarx)$ and that $\cpx$ is the mid-point of the segment connecting $\bm{x}$ and $r(\bm{x})$.
Substituting \eqref{eqn:central_diff} into \eqref{eqn:boundary_derivative}, we obtain the extrapolation formula
and the error term.
\end{proof}

\begin{rem}
In practice, $\| \bm{x} - \cpx\|$ is bounded by the radius of tubular neighbourhood, which is a small multiple of the mesh size $\Delta x$ \cite{ruuth2008simple, macdonald2010implicit}.
Hence, the error term in Theorem~\ref{thm:BCext} is $O(\Delta x^3)$.
\end{rem}

We can combine the extension condition $\bar{u} = \bar{E}\bar{u}$ in $\I(\S)$ with the extrapolation formula \eqref{eqn:extrap} to define an extension in all of $\B(\S)$:
\begin{equation}\label{eqn:extension}
\bar{u}(\bm{x}) = \bar{E}\,\bar{u}(\bm{x}) + 2\,\chi_{\E(\S)}(\bm{x}) \, \langle \bm{x} - \cpx, \mathbf{n}\rangle \, j\big(\cpx, E\,\bar{u}(\bm{x})\big), \quad \bm{x} \in \B(\S),
\end{equation}
where $\chi_{\E(\S)}$ is the indicator function of $\E(\S)$. 


\subsection{Embedding formulation and discretization}


We illustrate how to incorporate the extrapolation method into the CPM through an elliptic example:
\begin{equation}\label{eqn:spoisson}
\Delta_\S u - c\,u = f \quad \text{in $\S$,} \qquad \mbox{subject to \eqref{eqn:gbc} in $\partial\S$,}
\end{equation}
where $\Delta_\S$ is the Laplace--Beltrami operator on $\S$.
Let $\bar{u}$ be the extension of the surface solution $u$ according to \eqref{eqn:extension}. By the Laplacian principles \cite{ruuth2008simple, chen2015closest}, the Cartesian Laplacian $\Delta$ applied to the extension $\bar{u}$ coincides with $\Delta_\S u$ on $\S$, i.e., $\Delta \bar{u} \big|_{\S} = \Delta_\S u$. It follows that $E(\Delta \bar{u}) = E(\Delta_\S u)$. Therefore, by applying the extension operator $E$ to \eqref{eqn:spoisson}, we obtain the embedding equation $ E(\Delta \bar{u}) - E (c u + f) = 0$. Using the idempotent property $E^2 = E$ \cite{marz2012calculus}, this can be rewritten in the compact form
\begin{equation}\label{eqn:spoisson_embedding}
E\left(\Delta \bar{u} - c\bar{u} - \bar{f}\right) = 0 \,, \quad \text{in } \B(\S),
\end{equation}
where $\bar{f} := Ef$.
Following \cite{von2013embedded, chen2015closest, iyaniwura2021optimization}, we combine \eqref{eqn:spoisson_embedding} and \eqref{eqn:extension} via a penalty formulation:
\begin{equation}\label{eqn:embedding_equation}
E\left(\Delta \bar{u} - c\,\bar{u} - \bar{f}\right) - \gamma\Big[\bar{u} - \bar{E}\bar{u} - 2 \chi_{\E(\S)}(\bm{x}) \langle \bm{x} - \cpx, \mathbf{n}\rangle j(\cpx, E \bar{u}(\bm{x}))\Big] = 0 \,, \quad \bm{x} \in \B(\S) \,,
\end{equation}
where $\gamma > 0$ is a penalty parameter. In practice, we take $\gamma = 2d/(\Delta x)^2$ with $d=3$ as the embedding space dimension.

Let $\mathsf{u} \in \mathbb{R}^m$ denote the numerical approximation of $\bar{u}$ at grid points $\bm{x}_i$, $i=1,\dots,m$. Let $\mathsf{E}, \bar{\mathsf{E}} \in \mathbb{R}^{m\times m}$ be the cubic interpolation matrices such that $\left[\mathsf{E} \mathsf{u}\right]_i \approx u(\cp(\bm{x}_i))$ and $\left[\bar{\mathsf{E}} \mathsf{u}\right]_i \approx u(\bar{\cp}(\bm{x}_i))$ (see \cite{macdonald2010implicit, chen2015closest}). The discretized embedding system of \eqref{eqn:embedding_equation} reads
\begin{equation}\label{eqn:embedding_equation_discretized}
\mathsf{E} (\mathsf{L} \mathsf{u} - c\mathsf{u} - \mathsf{f}) - \gamma (\mathsf{u} - \bar{\mathsf{E}} \mathsf{u} - \mathsf{D}\,\mathsf{j}) = 0,
\end{equation}
where $\mathsf{D} \in \mathbb{R}^{m\times m}$ is the diagonal matrix with diagonal entry $\mathsf{D}_{i,i} = 2 \chi_{\E(\S)}(\bm{x}_i) \langle \bm{x}_i - \cp(\bm{x}_i), \mathbf{n}_i \rangle$, and $\mathsf{f}, \mathsf{j}\in\mathbb{R}^m$ are defined by $\mathsf{f}_i = f(\cp(\bm{x}_i))$ and $\mathsf{j}_i = j(\cp(\bm{x}_i), [\mathsf{E} \mathsf{u}]_i)$, respectively. 

\begin{rem}\label{rem:conormal_approximation}
When the analytical form of the unit co-normal $\mathbf{n}$ is unavailable, it can be approximated using surface reconstruction techniques \cite{ye2012multigrid, zhao2000implicit, lee2017revisiting}.
Alternatively, for $\bm{x}_i \in \E(\S)$, we approximate
$\mathbf{n}(\bm{x}_i)$ by $\frac{\cp(\bm{x}_i) - \bar{\cp}(\bm{x}_i)}
{\|\cp(\bm{x}_i) - \bar{\cp}(\bm{x}_i)\|}$,
which we justify as follows. By Taylor expansion,
$$
\bar{\cp}(\bm{x}_i) = \cp\!\left(\cp(\bm{x}_i) + (\cp(\bm{x}_i) - \bm{x}_i)\right)
= \cp(\bm{x}_i) + D\cp|_{\cp(\bm{x}_i)}(\cp(\bm{x}_i) - \bm{x}_i)
+ O(\|\bm{x}_i - \cp(\bm{x}_i)\|^2),
$$
where $D\cp$ is the Jacobian of $\cp$. Since $D\cp$ projects onto the tangent space of $\S$ \cite{marz2012calculus}, and $\bm{x}_i - \cp(\bm{x}_i)$ is orthogonal to the boundary tangent $\mathbf{T}(\bm{x}_i)$ (by Lemma~\ref{lem:span}), the vector $D\cp|_{\cp(\bm{x}_i)}(\bm{x}_i - \cp(\bm{x}_i))$ is parallel to $\mathbf{n}(\bm{x}_i)$. Hence, the above formula approximates $\mathbf{n}(\bm{x}_i)$ up to 
$O(\|\bm{x}_i - \cp(\bm{x}_i)\|^2)$. 
If $\|\cp(\bm{x}_i) - \bar{\cp}(\bm{x}_i)\|$ falls below a small threshold (taken as $10^{-4}$ in our experiments), indicating $\bm{x}_i - \cp(\bm{x}_i)$ is almost orthogonal to $\mathbf{n}(\bm{x}_i)$,
we set $\mathsf{D}_{i,i} = 0$ to avoid division-by-zero.

\end{rem}

\section{Numerical examples}

In this section, we present numerical examples to validate the extrapolation method in CPM. All computations were performed in \textsc{Matlab} R2024a; linear systems are solved using the \texttt{backslash} operator, and the eigenvalues are computed with the \texttt{eigs} function.
For the computation of the outward normal $\mathbf{n}$, we used the approximation based on $\cp(\bm{x}) - \bar{\cp}(\bm{x})$ introduced in Remark~\ref{rem:conormal_approximation}.


\subsection{Poisson equation}

Consider the surface Poisson equation with Robin condition,
\begin{equation}\label{eqn:poisson}
\Delta_{\mathcal{S}} u = f \quad \text{on } \mathcal{S}, 
\qquad \partial_n u + \kappa u = g \quad \text{on } \partial\mathcal{S}.
\end{equation}
This is a special case of \eqref{eqn:spoisson} with $c=0$ and $j(\bm{x},u) = -\kappa u + g(\bm{x})$, which we substitute into \eqref{eqn:embedding_equation_discretized} to obtain the linear system
\begin{equation}\label{eqn:poisson_numerical}
\mathsf{Au = b} ,, \qquad \text{with } \mathsf{A} = \bar{\mathsf{E}} \mathsf{L} - \gamma \left( \mathsf{I} - \bar{\mathsf{E}} + \kappa \mathsf{D}\mathsf{E} \right), \quad
\mathsf{b} = \bar{\mathsf{E}} \mathsf{f} - \gamma \mathsf{g} \,,
\end{equation}
where $\mathsf{g} \in \mathbb{R}^m$ with $\mathsf{g}_i = g(\bm{x}_i)$. For the convergence study, we prescribe the exact solution
\begin{equation}\label{eqn:poisson_solution}
u(\phi,\theta)=\cos(2\phi)\sin^2\theta + \sin(3\phi)\sin^3\theta
\end{equation}
on the unit upper hemisphere $\S = \left\{ (\phi,\theta) : \phi \in [0,2\pi], \,\, \theta \in [0,\pi/2] \right\}$. One can verify that \eqref{eqn:poisson_solution} is the Poisson solution to \eqref{eqn:poisson} with $\kappa=1$,
$f(\phi,\theta)=-6\sin^2\theta \cos(2\phi)-12\sin^3\theta \sin(3\phi)$, and $g(\phi)=\cos(2\phi)+\sin(3\phi)$ on $\partial\mathcal{S} = \{ (\phi,\theta): \theta=\pi/2\}$. The numerical solution $\mathsf{u}$ obtained from \eqref{eqn:poisson_numerical} is compared against the exact solution \eqref{eqn:poisson_solution} with successively refined mesh size $\Delta x$. As shown in Table~\ref{tab:robin_conv}, CPM with proposed extrapolation method achieves second-order accuracy.

\begin{table}[t]
  \centering
  \caption{%
    Numerical convergence of the Poisson problem \eqref{eqn:poisson}.
  }
  \label{tab:robin_conv}
  \begin{tabular}{cccc}
    \toprule
    $\Delta x$ & Length($\mathsf{u}$) & Relative error in max--norm & Order \\
    \midrule
    0.1    & 7,161   & 5.4284e-3 & --     \\
    0.05   & 24,321  & 1.2647e-3 & 2.1017 \\
    0.025  & 89,989  & 3.0515e-4 & 2.0512 \\
    0.0125 & 345,297 & 7.7049e-5 & 1.9857 \\
    \bottomrule
  \end{tabular}
\end{table}

\subsection{Steklov eigenvalue problem}
Next, we consider the Steklov eigenvalue problem \cite{girouard2017spectral}, which we compute nontrivial eigenpairs $(\phi,\sigma)$ that satisfy
\begin{equation}\label{eqn:steklov}
    \Delta_\S \phi = 0 \quad \text{in } \S, 
    \qquad \partial_n \phi = \sigma \phi \quad \text{on } \partial\S.
\end{equation}
This problem arises as a special case of \eqref{eqn:spoisson} with $c=0$, $f \equiv 0$, and $j(\bm{x},\phi) = \sigma \phi$. The corresponding discretized embedding formulation of \eqref{eqn:steklov} is given by
    $\bar{\mathsf{E}} (\mathsf{L \phi})  
    - \gamma \left[ \mathsf{\phi} - \bar{\mathsf{E}} \mathsf{\phi} - \mathsf{D}(\sigma \, \mathsf{E} \mathsf{\phi}) \right] = \mathsf{0},$
    which leads to the generalized eigenvalue problem $\mathsf{A u} = \sigma \mathsf{B u}$, where $\mathsf{A} = \bar{\mathsf{E}} \mathsf{L} - \gamma(\mathsf{I}-\bar{\mathsf{E}})$ and $\mathsf{B} = -\gamma \, \mathsf{D}\mathsf{E}$. We show the computed eigenfunctions for the unit hemisphere and M\"obius strip in Figure~\ref{fig:steklov_combined}.
    A conformal mapping argument shows that the Steklov spectrum of the unit hemisphere coincides with that of the unit disk, whose eigenvalues are $\sigma_0 = 0$ and $\sigma_n = n$ (with multiplicity two) for $n \geq 1$ \cite{girouard2017spectral} . We demonstrate in Figure~\ref{fig:steklov_combined} that our method exhibits second-order accuracy for a selected set of eigenvalues. 

\begin{figure}[thb]
    \centering
    \includegraphics[width=0.825\linewidth]{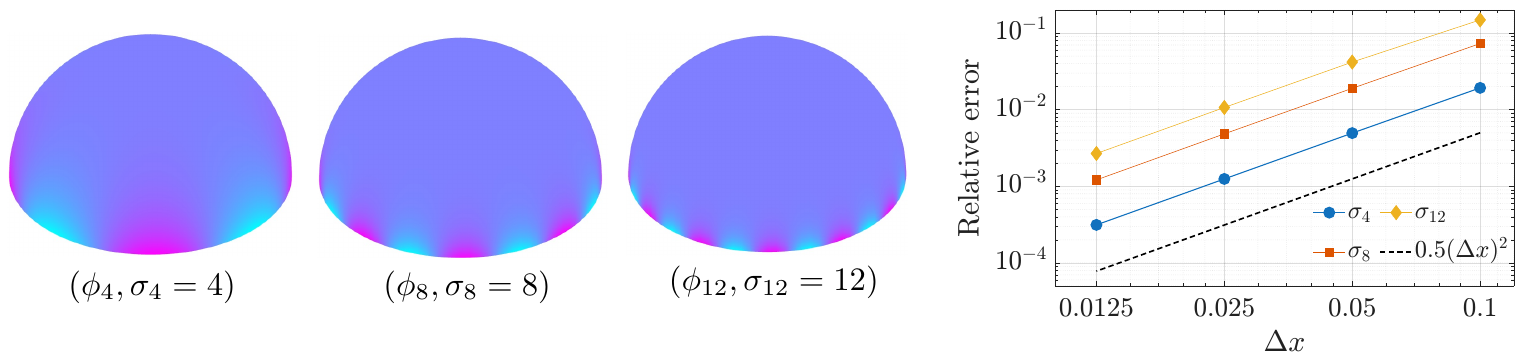}
            
    \vspace{3mm}
    \includegraphics[width=0.9\linewidth]{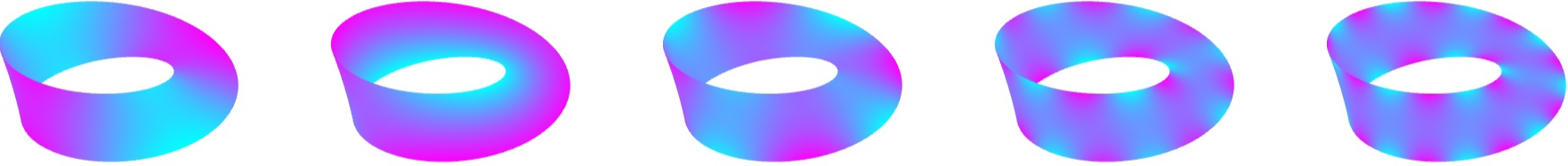}
    
    \caption{\textbf{Steklov eigenfunctions for different geometries}. Top: Unit hemisphere, a selection of the Steklov eigenfunctions, numerically computed with mesh size $\Delta x = 0.025$ ($89,989$ grid points). The rightmost plot shows the second-order accuracy of these selected eigenvalues. Bottom: M\"obius strip, a selection of the Steklov eigenfunctions, numerically computed with mesh size $\Delta x = 0.025$ ($72,062$ grid points).}
    \label{fig:steklov_combined}
\end{figure}

\subsection{Gray--Scott reaction-diffusion system}
Finally, we use our proposed numerical method to demonstrate how the temporal dynamics of the Gray--Scott pattern formation system
\begin{equation*}
\partial_t u = D_u \Delta_{\mathcal{S}} u - u v^2 + F(1 - u) \,, \quad
\partial_t v = D_v \Delta_{\mathcal{S}} v + u v^2 - (F + k)v \,, \quad \text{in $\S$},
\end{equation*}
are affected by Robin-type boundary conditions
$\partial_n u + \kappa u = 0, \partial_n v + \kappa v = 0$ on $\partial \S$. The Robin condition represents leakage of the concentrations $u$ and $v$ through the boundary, controlled by the leakage parameter $\kappa > 0$.
We choose $\S$ to be a M\"obius strip and set the parameters to $F = 0.010$, $k = 0.042$, $D_u = 8\times10^{-5}$, and $D_v = 0.4 D_u$.
Figure~\ref{fig:GS} shows that the system exhibits weak patterns that dissipate in finite time under no-flux conditions ($\kappa = 0$; homogeneous Neumann condition),
whereas the case $\kappa = 10$ induces chaotic and persistent spatiotemporal patterns.


\begin{figure}[htb]
    \centering
    \includegraphics[width=0.9\linewidth]{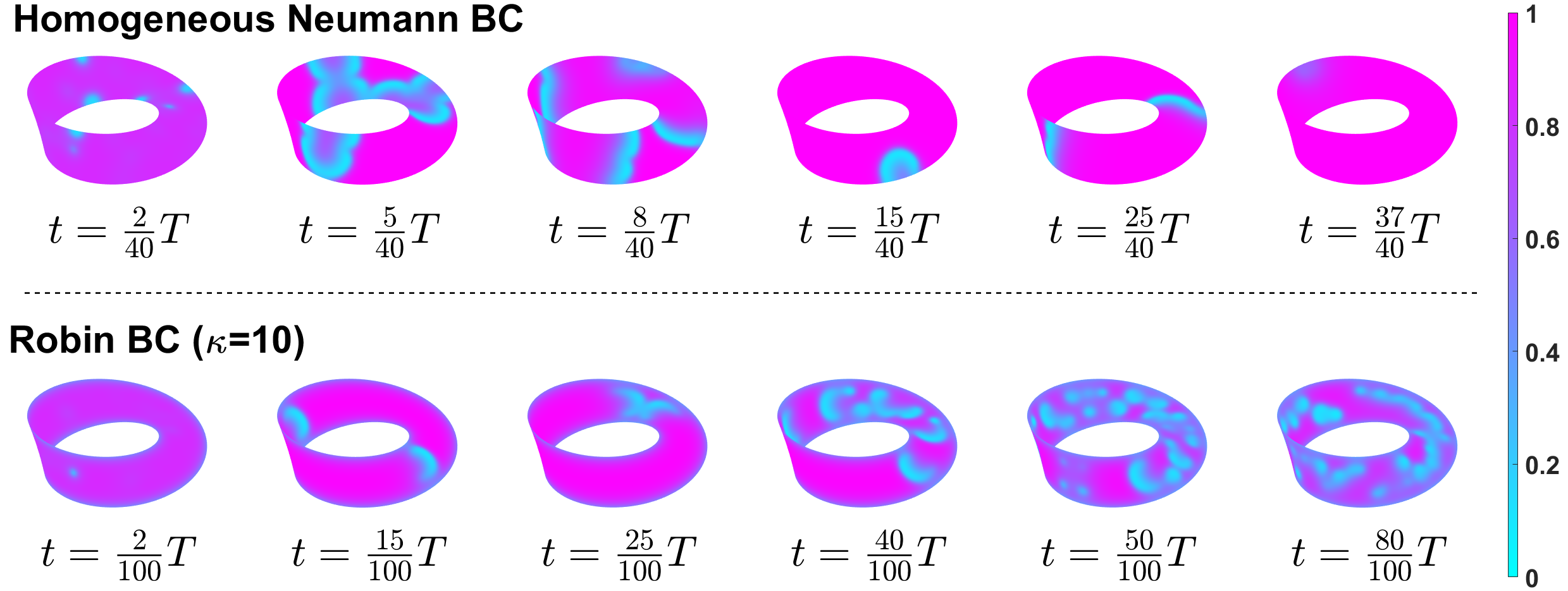}
    \caption{\textbf{Gray-Scott model with Robin boundary condition}. Numerical solution of the substrate $u$ on the surface of the M\"obius strip at different time points. Top row: $\kappa=0$ (homogeneous Neumann condition), with total simulation time $T=4000$. Bottom row: $\kappa=10$, $T=10000$.}
    \label{fig:GS}
\end{figure}


\printbibliography

@article{ruuth2008simple,
  title={A simple embedding method for solving partial differential equations on surfaces},
  author={Ruuth, Steven J and Merriman, Barry},
  journal={J. Comput. Phys.},
  volume={227},
  number={3},
  optpages={1943--1961},
  year={2008},
  publisher={Elsevier},
  doi={10.1016/j.jcp.2007.10.009}
}

@article{macdonald2011solving,
  title={Solving eigenvalue problems on curved surfaces using the closest point method},
  author={Macdonald, Colin B and Brandman, Jeremy and Ruuth, Steven J},
  journal={J. Comput. Phys.},
  volume={230},
  number={22},
  optpages={7944--7956},
  year={2011},
  publisher={Elsevier},
  doi={10.1016/j.jcp.2011.06.021}
}

@article{macdonald2010implicit,
  title={The implicit closest point method for the numerical solution of partial differential equations on surfaces},
  author={Macdonald, Colin B and Ruuth, Steven J},
  journal={SIAM J. Sci. Comput.},
  volume={31},
  number={6},
  optpages={4330--4350},
  year={2010},
  publisher={SIAM},
  doi={10.1137/080740003}
}

@article{macdonald2013simple,
  title={Simple computation of reaction-diffusion processes on point clouds},
  author={Macdonald, Colin B and Merriman, Barry and Ruuth, Steven J},
  journal={Proc. Natl. Acad. Sci. U.S.A.},
  volume={110},
  number={23},
  optpages={9209--9214},
  year={2013},
  publisher={National Academy of Sciences},
  doi={10.1073/pnas.1221408110}
}

@article{chen2015closest,
  title={The closest point method and multigrid solvers for elliptic equations on surfaces},
  author={Chen, Yujia and Macdonald, Colin B},
  journal={SIAM J. Sci. Comput.},
  volume={37},
  number={1},
  optpages={A134--A155},
  year={2015},
  publisher={SIAM},
  doi={10.1137/130929497}
}

@article{von2013embedded,
title={An embedded method-of-lines approach to solving partial differential equations on surfaces},
author={von Glehn, Ingrid and M{\"a}rz, Thomas and Macdonald, Colin B},
journal={arXiv preprint arXiv:1307.5657},
year={2013},  
}

@article{charette2020pattern,
  title={Pattern formation in a slowly flattening spherical cap: delayed bifurcation},
  author={Charette, Laurent and Macdonald, Colin B and Nagata, Wayne},
  journal={IMA J. Appl. Math.	},
  volume={85},
  number={4},
  optpages={513--541},
  year={2020},
  publisher={Oxford University Press},
  doi={10.1093/imamat/hxaa016}
}

@article{rozada2014stability,
  title={The stability of localized spot patterns for the {B}russelator on the sphere},
  author={Rozada, Ignacio and Ruuth, Steven J and Ward, Michael Jeffrey},
  journal={SIAM J. Appl. Dyn. Syst.},
  volume={13},
  number={1},
  optpages={564--627},
  year={2014},
  publisher={SIAM},
  doi={10.1137/130934696}
}

@article{marz2016embedding,
  title={An embedding technique for the solution of reaction--diffusion equations on algebraic surfaces with isolated singularities},
  author={M{\"a}rz, Thomas and Rockstroh, Parousia and Ruuth, Steven J},
  journal={J. Math. Anal. Appl.},
  volume={436},
  number={2},
  optpages={911--943},
  year={2016},
  publisher={Elsevier},
  doi={10.1016/j.jmaa.2015.11.064}
}

@inproceedings{tian2009segmentation,
  title={Segmentation on surfaces with the closest point method},
  author={Tian, Li and Macdonald, Colin B and Ruuth, Steven J},
  booktitle={2009 16th IEEE International Conference on Image Processing (ICIP)},
  optpages={3009--3012},
  year={2009},
  organization={IEEE},
  doi={10.1109/ICIP.2009.5414447}
}

@inproceedings{auer2012real,
  title={Real-time fluid effects on surfaces using the closest point method},
  author={Auer, Stefan and Macdonald, Colin B and Treib, Marc and Schneider, Jens and Westermann, R{\"u}diger},
  booktitle={Comput. Graphics Forum},
  volume={31},
  number={6},
  optpages={1909--1923},
  year={2012},
  organization={Wiley Online Library},
  doi={10.1111/j.1467-8659.2012.03071.x}
}

@inproceedings{biddle2013volume,
  title={A volume-based method for denoising on curved surfaces},
  author={Biddle, Harry and von Glehn, Ingrid and Macdonald, Colin B and M{\"a}rz, Thomas},
  booktitle={2013 IEEE International Conference on Image Processing},
  optpages={529--533},
  year={2013},
  organization={IEEE},
  doi={10.1109/ICIP.2013.6738109}
}

@article{king2024closest,
  title={A Closest Point Method for {PDE}s on manifolds with interior boundary conditions for geometry processing},
  author={King, Nathan and Su, Haozhe and Aanjaneya, Mridul and Ruuth, Steven J and Batty, Christopher},
  journal={ACM Trans. Graphics},
  volume={43},
  number={5},
  optpages={1--26},
  year={2024},
  publisher={ACM New York, NY, USA}, 
  doi={10.1145/3673652}
}

@inproceedings{arteaga2015laplace,
  title={Laplace--{B}eltrami spectra for shape comparison of surfaces in {3D} using the closest point method},
  author={Arteaga, Reynaldo J and Ruuth, Steven J},
  booktitle={2015 IEEE International Conference on Image Processing (ICIP)},
  optpages={4511--4515},
  year={2015},
  organization={IEEE},
  doi={10.1109/ICIP.2015.7351660}
}

@article{iyaniwura2021simulation,
  title={Simulation and optimization of mean first passage time problems in 2-D using numerical embedded methods and perturbation theory},
  author={Iyaniwura, Sarafa and Wong, Tony and Ward, Michael J and Macdonald, Colin B},
  journal={Multiscale Model. Simul.},
  volume={19},
  number={3},
  optpages={1367--1393},
  year={2021},
  publisher={SIAM},
  doi={10.1137/19M1299621}
}

@article{iyaniwura2021optimization,
  title={Optimization of the mean first passage time in near-disk and elliptical domains in 2-D with small absorbing traps},
  author={Iyaniwura, Sarafa A and Wong, Tony and Macdonald, Colin B and Ward, Michael J},
  journal={SIAM Rev.},
  volume={63},
  number={3},
  optpages={525--555},
  year={2021},
  publisher={SIAM},
  doi={10.1137/20M1332396}
}

@book{hirsch2012differential,
  title={Differential topology},
  author={Hirsch, Morris W},
  volume={33},
  year={2012},
  publisher={Springer Science \& Business Media}
}

@article{girouard2017spectral,
  title={Spectral geometry of the {S}teklov problem (survey article)},
  author={Girouard, Alexandre and Polterovich, Iosif},
  journal={J. Spectral Theory},
  volume={7},
  number={2},
  optpages={321--359},
  year={2017},
  doi={10.4171/JST/164}
}

@article{marz2012calculus,
  title={Calculus on surfaces with general closest point functions},
  author={Marz, Thomas and Macdonald, Colin B},
  journal={SIAM J. Numer. Anal.},
  volume={50},
  number={6},
  optpages={3303--3328},
  year={2012},
  publisher={SIAM},
  doi={10.1137/12086553}
}

@article{zhao2000implicit,
  title={Implicit and nonparametric shape reconstruction from unorganized data using a variational level set method},
  author={Zhao, Hong-Kai and Osher, Stanley and Merriman, Barry and Kang, Myungjoo},
  journal={Comput. Vision Image Understanding},
  volume={80},
  number={3},
  optpages={295--314},
  year={2000},
  publisher={Elsevier},
  doi={10.1006/cviu.2000.0875}
}

@inproceedings{ye2012multigrid,
  title={Multigrid narrow band surface reconstruction via level set functions},
  author={Ye, Jian and Yanovsky, Igor and Dong, Bin and Gandlin, Rima and Brandt, Achi and Osher, Stanley},
  booktitle={International Symposium on Visual Computing},
  optpages={61--70},
  year={2012},
  organization={Springer},
  doi={10.1007/978-3-642-33179-4_7}
}

@article{lee2017revisiting,
  title={Revisiting the redistancing problem using the {H}opf--{L}ax formula},
  author={Lee, Byungjoon and Darbon, Jerome and Osher, Stanley and Kang, Myungjoo},
  journal={J. Comput. Phys.},
  volume={330},
  optpages={268--281},
  year={2017},
  publisher={Elsevier},
  doi={10.1016/j.jcp.2016.11.005}
}

\end{document}